\RequirePackage{iftex}
\documentclass[11pt,a4paper]{amsart}
\ifptex
	\RequirePackage{plautopatch}
	\usepackage[dvipdfmx]{graphicx}
\fi
\usepackage[alphabetic]{amsrefs}
\usepackage{tikz,tikz-cd}
\usepackage{amsmath,amssymb,amsthm}
\usepackage{mathtools}
\usepackage{mathrsfs}
\usepackage{hyperref}
\hypersetup{
	pdftitle={A NOTE ON WEAKLY PROREGULAR SEQUENCES AND LOCAL COHOMOLOGY},
	pdfauthor={Ryoya Ando}
}
\usepackage{myarticlepreamble}
\theoremstyleenglish

\title
{A NOTE ON WEAKLY PROREGULAR SEQUENCES AND LOCAL COHOMOLOGY}

\author{Ryoya Ando}
\date{\today}
\address{\tusaddressfull}
\email{\gmail}
\email{\tusmail}
\subjclass[2020]{13D03 (Primary) 13D45 (Secondary)}

\begin{document}

\begin{abstract}
	In this note, we give an elementary proof of the result given by Schenzel that there are functorial isomorphisms between local cohomology groups and \v{C}ech cohomology groups, by using weakly proregular sequences.
	In \cite{Schenzel}, he used notions of derived category theory in his proof, but we do not use them in this paper.
\end{abstract}

\maketitle

\section{Introduction} 
In this note, we assume all rings are commutative with identity element. Let $A$ denote a ring and $I$ an ideal of $A$. The functor $\Gamma_I$ is defined by
\[\Gamma_I(M)=\mkset{x\in M}{\text{$I^nx=0$ for some $n \geq 0$}}\]
for an $A$-module $M$. Here $H^i_I(-)$ denote the local cohomology functors, defined as the right derived functors of $\Gamma_I(-)$. In other words, let $J^\bullet$ be an injective resolution of the $A$-module $M$, then $H^i_I(M)\cong H^i(\Gamma_I(J^\bullet))$. In Noetherian cases, the local cohomology can be written by using \v{C}ech cohomology. Let $\underline{a}=a_1,\dots, a_r$ be a sequence of elements of $A$, and $I=(a_1,\dots,a_r)$. $\check{H}^i(\underline{a},M)$ denote the \v{C}ech cohomology (see \ref{defi:CechCohomology}). It is well-known that there are isomorphisms;
\[H^i_I(M)\cong \check{H}^i(\underline{a},M)\tag{$\ast$}\]
for any $A$-module $M$ if $A$ is a Noetherian ring and $I=(a_1,\dots,a_r)$, see for example \cite[Theorem 3.5.6.]{BrunsHerzog}.

This result was generalised by \cite{Schenzel}. For arbitrary ring $A$, he showed that formula $(\ast)$ is true for any $A$-module $M$ if and only if $\underline{a}$ is a weakly proregular sequence. Let $H_i(\underline{a})$ be the Koszul homology of the sequence $\underline{a}$. A system of elements $\underline{a}=a_1,\dots,a_r$ is called a weakly proregular sequence if for any $1\leq i\leq d$ and for each $n>0$ there is an $m\geq n$ such that the natural map;
\[H_i(\underline{a}^m)\to H_i(\underline{a}^n)\]
is zero map, note that $\underline{a}^n$ is the sequence defined by $a_1^n,\dots,a_r^n$.

The goal of this note is the explanation of the following result without using notions of derived category theory.
\begin{thm}[\ref{thm:goal}]
	Let $A$ be a ring, $\underline{a}=a_1,\dots,a_r$ a system of elements of $A$ and $I=(a_1,\dots,a_r)$. $\underline{a}$ is a weakly proregular sequence if and only if for any $i$ and $A$-module $M$, $H^i_I(M)\cong\check{H}^i(\underline{a},M)$.
\end{thm}

In section 2, we summarise the $\delta$-functors. It is also mentioned in \cite[Chap.3.1]{Hartshorne} without proof. And in section 3 we resume the definition of \v{C}ech cohomology in commutative algebra. In section 4, we present the theory of weakly proregular sequences, following \cite[Sect.2]{Schenzel}. Finally, We introduce the main results in section 5.

\section{$\delta$-functors}

\begin{defi}
	Let $\mathscr{A},\mathscr{B}$ be an Abelian categories, and $\mathscr{A}$ has enough injectives. Let  $F\colon \mathscr{A}\to\mathscr{B}$ be an additive left exact functor. $I^\bullet$ denote an injective resolution of $A\in\mathscr{A}$. The functor;
	\[R^iF\colon\mathscr{A}\to\mathscr{B};A\mapsto H^i(F(I^\bullet))\]
	is called the \textbf{right derived functor} of $F$.
\end{defi}

Note that derived functors are independent up to natural transformation of the selection of injective resolution. The following is a characteristic property of the derived functor.

\begin{prop}\label{prop:prop_of_derived_functor}
	Let $\mathscr{A},\mathscr{B}$ be Abelian categories, and assume $\mathscr{A}$ has enough injectives. Let  $F\colon \mathscr{A}\to\mathscr{B}$ be additive left exact functor. Then
	\begin{enumerate}
		\item $R^0F\cong F$~(as functors).
		\item For any exact sequence in $\mathscr{A}$;
		\[\ses[f][g]{A_1}{A_2}{A_3}\]
		for each $i\geq0$, there are \textbf{connecting morphisms} $\delta^i\colon R^iF(A_3)\to R^{i+1}F(A_1)$ such that;
		\[\begin{tikzcd}[row sep=tiny, column sep=scriptsize]
			0\nxcell F(A_1)\nxcell[F(f)]F(A_2)\nxcell[F(g)]F(A_3)\nxcell[\delta^0]{}\cdots\\
			{}\nxcell[\delta^{i-1}]R^iF(A_1)\nxcell[R^iF(f)]R^iF(A_2)\nxcell[R^iF(g)]R^iF(A_3)\nxcell[\delta^i]\cdots
		\end{tikzcd}\]
		is exact sequence in $\mathscr{B}$.
		\item Given a commutative diagram in $\mathscr{A}$ of the form (where the rows are exact);
		\[\begin{tikzcd}
			0\nxcell A_1\arrow[d,"\gamma"]\nxcell[f]A_2\arrow[d,"\beta"]\nxcell[g]A_3\arrow[d,"\alpha"]\nxcell0\\
			0\nxcell B_1\nxcell[f']B_2\nxcell[g']B_3\nxcell0
		\end{tikzcd}\]
		for any $i\geq0$ the following diagram;
		\[\begin{tikzcd}
			R^iF(A_3)\arrow[d,"R^iF(\alpha)"]\nxcell[]R^{i+1}F(A_1)\arrow[d,"R^{i+1}F(\gamma)"]\\
			R^iF(B_3)\nxcell[]R^{i+1}F(B_1)
		\end{tikzcd}\]
		 is commutative in $\mathscr{B}$.
		\item For each injective object $I\in\mathscr{A}$, for any $i>0$, $R^iF(I)=0$.
	\end{enumerate}
\end{prop}

See a textbook on homological algebra for the proof of this proposition. The $\delta$-functor can be thought of as an extract of the above property.

\begin{defi}
	Let $\mathscr{A},\mathscr{B}$ be an Abelian categories. Families of additive functors $T^\bullet\coloneqq\{T^i\}$ are called \textbf{$\delta$-functors} when the following conditions hold;
	\begin{enumerate}
		\item For any exact sequence in $\mathscr{A}$;
		\[\ses[f][g]{A_1}{A_2}{A_3}\]
		for each $i\geq0$, there are \textbf{connecting morphism} $\delta^i\colon T^i(A_3)\to T^{i+1}(A_1)$ such that;
		\[\begin{tikzcd}[row sep=tiny, column sep=scriptsize]
			0\nxcell T^0(A_1)\nxcell[T^0(f)]T^0(A_2)\nxcell[T^0(g)]T^0(A_3)\nxcell[\delta^0]{}\cdots\\
			{}\nxcell[\delta^{i-1}]T^i(A_1)\nxcell[T^i(f)]T^i(A_2)\nxcell[T^i(g)]T^i(A_3)\nxcell[\delta^i]\cdots
		\end{tikzcd}\]
		is exact sequence in $\mathscr{B}$.
		\item Given a commutative diagram in $\mathscr{A}$ of the form (where the rows are exact);
		\[\begin{tikzcd}
			0\nxcell A_1\arrow[d,"\alpha"]\nxcell[f]A_2\arrow[d,"\beta"]\nxcell[g]A_3\arrow[d,"\gamma"]\nxcell0\\
			0\nxcell B_1\nxcell[f']B_2\nxcell[g']B_3\nxcell0
		\end{tikzcd}\]
		then for any $i\geq0$ the following diagram;
		\[\begin{tikzcd}
			T^i(A_3)\arrow[d,"T^i(\gamma)"]\nxcell[]T^{i+1}(A_1)\arrow[d,"T^{i+1}(\alpha)"]\\
			T^i(B_3)\nxcell[]T^{i+1}(B_1)
		\end{tikzcd}\]
		is commutative in $\mathscr{B}$.
	\end{enumerate} 
\end{defi}

Let us define that two $\delta$-functors are isomorphic in the following way. Let $T^\bullet, U^\bullet$ be a $\delta$-functor. Families of natural transformation $\theta^\bullet=\{\theta^i\colon T^i\Rightarrow U^i\}$ is called \textbf{morphism of $\delta$-fanctors} when for each exact sequence ;
\[\ses{A_1}{A_2}{A_3}\]
in $\mathscr{A}$, the following diagram;
\[\begin{tikzcd}
	T^i(A_3)\darrow[\theta^i_{A_3}]\nxcell[\delta^i_T]T^{i+1}(A_1)\darrow[\theta^{i+1}_{A_1}]\\
	U^i(A_3)\nxcell[\delta^i_U]U^{i+1}(A_1)
\end{tikzcd}\]
is commutative. An isomorphism is a morphism which has a two-sided inverse.

\begin{defi}
	Let $\mathscr{A},\mathscr{B}$ be Abelian categories. The $\delta$-functor $T^\bullet$ is called \textbf{universal} when for each $\delta$-functor $U^\bullet$ and natural transformation $\theta\colon T^0\Rightarrow U^0$, there is a unique morphism of $\delta$-functors $\theta^\bullet\colon T^\bullet\to U^\bullet$ such that $\theta^0=\theta$. 
\end{defi}

By definition, two universal $\delta$-functors such that $T^0=U^0$ are isomorphic up to unique isomorphism. So for each additive functor $F\colon \mathscr{A}\to\mathscr{B}$, a universal $\delta$-functor $T^\bullet$ with $T^0=F$ is unique up to unique isomorphism if it exists. A universal $\delta$-functor $T^\bullet$ with such a property is called a right satellite functor of $F$.

The following property gives a condition for the $\delta$-functor to be universal, which shows that the derived functor is so if $\mathscr{A}$ has enough injectives.

\begin{defiprop}\label{prop:effaceable->universal}
	Let $\mathscr{A},\mathscr{B}$ be Abelian categories, $F$ an additive functor. $F$ said to be \textbf{effaceable} if for each $A\in\mathscr{A}$, there are $M\in\mathscr{A}$ and injection (monomorphism) $u\colon A\to M$ such that $F(u)=0$. A $\delta$-functor $T^\bullet$ is universal if for each $i>0$, $T^i$ is effaceable.
\end{defiprop}

\begin{proof}
	Let $U^\bullet$ be $\delta$-functor and $\theta\colon T^0\Rightarrow U^0$ natural transformation. We show that there exists uniquely a morphism of $\delta$-functor $\theta^\bullet\colon T^\bullet\to U^\bullet$ such that $\theta^0=\theta$. We construct it inductively. For any $A\in\mathscr{A}$, there is injection $u\colon A\to M$ such that $T^1(u)=0$ since $T^1$ is effaceable. Let $C$ be the cokernel of $u$. We consider the long exact sequences induced by $\ses[u][\pi]{A}{M}{C}$. So we get the following commutative diagram.
	\[\begin{tikzcd}
		T^0(M)\darrow[\theta_M]\nxcell[T^0(\pi)]\darrow[\theta_C]T^0(C)\nxcell[\delta^0_T]T^1(A)\arrow[d,dashed]\nxcell[T^1(u)=0] 0\\
		U^0(M)\nxcell[U^0(\pi)]U^0(C)\nxcell[\delta^0_U]U^1(A)
	\end{tikzcd}\]

	Now $\theta^1_{A,u}\coloneqq\delta^0_U\circ\theta_C\circ(\delta^0_T)^{-1}\colon T^1(A)\to U^1(A)$ is well-defined since the rows are exact. We show $\theta^1_{A,u}$ is independent of the choice of $u$. Let $u'\colon A\to M'$ be an injection such that $T^1(u')=0$. $M\sqcup_A M'$ denote co-fibre products of $M$ and $M'$ on $A$. Then we get an injection $u''\colon A\to M\sqcup M'$ such that $T^1(u'')=0$. Let $C''$ be the cokernel of $u''$. Then the following diagram;
	\[\begin{tikzcd}
		&T^0(C)\arrow[rr]\arrow[dd,"\theta_{C''}",near start]&&T^1(A)\arrow[rr,crossing over]\arrow[dd,"\theta_{A,u}^1",near start]&&0\\
		T^0(C'')\arrow[from={ur}]\arrow[dd,"\theta_{C''}",near start]\arrow[rr, crossing over]&&T^1(A)\arrow[ur,equal]\arrow[rr, crossing over ]&&0\\
		&U^0(C)\arrow[rr]&&U^1(A)\\
		U^0(C'')\arrow[from={ur}]\arrow[rr]&&U^1(A)\arrow[ur,equal]\arrow[from={uu},"\theta_{A,u''}^1",near start, crossing over]
	\end{tikzcd}\] 
	is commutative. So we have $\theta_{A,u}^1=\theta_{A,u''}^1$, similarly we obtain $\theta_{A,u'}^1=\theta_{A,u''}^1$, then we get $\theta_{A,u}^1=\theta_{A,u'}^1$. So $\theta_A^1$ is independent of the choice of $u$.
	
	Secondly, we show that for each $f\in\hom_{\mathscr{A}}(A,B)$ the following diagram;
	\[\begin{tikzcd}
		T^1(A)\darrow[\theta_A^1]\nxcell[T^1(f)]T^1(B)\darrow[\theta_B^1]\\
		U^1(A)\nxcell[U^1(f)]U^1(B)
	\end{tikzcd}\]
	is commutative to prove $\theta^1$ is a natural transformation. For any injections $u\colon A\to M$ and $v\colon B\to N$ with $T^1(u)=T^1(v)=0$, we take co-fibre product;
	\[\begin{tikzcd}
		A\darrow[v\circ f]\nxcell[u]M\darrow\\
		N\nxcell[u']M\sqcup_A N
	\end{tikzcd}\]
    then $u'$ is injective. So we have an injection $u'\circ v\colon B\to M\sqcup N$ with $T^1(u'\circ v)=0$. Then we replace $N$ by $M\sqcup N$ and get the following commutative diagram with exact rows;
    	\[\begin{tikzcd}
    	0\nxcell A\darrow[f]\nxcell M\darrow\nxcell C\arrow[d,dashed]\nxcell 0\\
    	0\nxcell B\nxcell N\nxcell C'\nxcell 0
    \end{tikzcd}\]
	So $\theta^1$ is a natural transformation since the following diagram;
	\[\begin{tikzcd}
		&T^0(C)\arrow[dl]\arrow[dd]\arrow[rr]&&T^1(A)\arrow[ld,"T^1(f)"]\arrow[dd,"\theta_A^1",near start]\arrow[rr]&&0\\
		T^0(C')\arrow[dd]\arrow[rr,crossing over]&&T^1(B)\arrow[rr,crossing over]&&0\\
		&U^0(C)\arrow[ld]\arrow[rr]&&U^1(A)\arrow[ld,"U^1(f)"]\\
		U^0(C')\arrow[rr,""]&&U^1(B)\arrow[from=uu,"\theta_B^1",near start,crossing over]
	\end{tikzcd}\]
	is commutative.
	
	Finally, we show that $\theta_A^1$ is commutative with the connecting morphism. Let $\ses{A_1}{A_2}{A_3}$ be the short exact sequence in $\mathscr{A}$, we use the same method as above for the injection $u\colon A_1\to M$ with $T^1(u)=0$ so that each row of the following commutative diagram is exact.
	\[\begin{tikzcd}
		0\nxcell A\arrow[d,equal]\nxcell A_2\darrow\nxcell A_3\darrow\nxcell0\\
		0\nxcell A_1\nxcell M\nxcell C\nxcell0
	\end{tikzcd}\]
	We consider the following diagram.
	\[\begin{tikzcd}
		&T^0(A_3)\arrow[ld,"\theta_{A_3}",swap]\arrow[rr,"\delta",near start]\arrow[dd]&&T^1(A_1)\arrow[ld,"\theta_{A_1}^1"]\arrow[dd,equal]\\
		U^0(A_3)\arrow[dd]\arrow[rr,"\delta",near start,crossing over]&&U^1(A_1)\\
		&T^0(C)\arrow[ld,"\theta_C",swap]\arrow[rr,"\delta",near start]&&T^1(A_1)\arrow[ld,"\theta_{A_1}^1"]\\
		U^0(C)\arrow[rr,"\delta",near start]&&U^1(A_1)\arrow[from=uu,equal,crossing over]
	\end{tikzcd}\]
	
	The desired commutativity of $\theta_{A_1}^1, \theta_{A_3}$ and $\delta$ follows from commutativity of other squares, which follow from the construction of $\theta_{A_1}^1$ and the facts that $T^\bullet$ and $U^\bullet$ are $\delta$-functor and that $\theta$ is a natural transformation. 
	
	This shows that $\theta^1$ is a natural transformation commutative with connecting morphism, and its uniqueness can be seen from its construction (the universality of cokernel). In this way $\theta^1$ to $\theta^2$ can be created and continued inductively.
\end{proof}

\begin{cor}
	Let $\mathscr{A}, \mathscr{B}$ be Abelian categories, and $\mathscr{A}$ has enough injectives. Let $T^\bullet\colon  \mathscr{A}\to\mathscr{B}$ be a universal $\delta$-functor, then $T^0$ is left-exact and for each $i\geq 0$ there is a natural isomorphism $T^i\cong R^iT^0$.
\end{cor}

\begin{proof}
	$T^0$ is left exact by the definition of $\delta$-functor, so there are right derived functors $R^iT^0$. For each $i>0$, $R^iT^0$ is effaceable by (4) of \ref{prop:prop_of_derived_functor}, then $R^\bullet T^0$ is a universal $\delta$-functor. Now $R^0T^0=T^0$, so there is a unique isomorphism $R^\bullet T^0\cong T^\bullet$ by universality. 
\end{proof}

\section{\v{C}ech cohomology and Koszul comoplex}

In this section, we review the \v{C}ech cohomology of rings and modules. Let $A$ be a ring and fix a sequence $a_1, \dots, a_r$ of elements of $A$. For each $I=\{j_1,\dots,j_i\}\subset\{1,\dots,r\}~(j_1<\dots<j_i)$, let $a_I=a_{j_1}\dots a_{j_i}$. $e_1,\dots, e_r $ denotes the standard basis of $A^r$ and let $e_I=e_{j_1}\wedge\dots\wedge e_{j_i}$. 

\begin{defi}\label{defi:CechCohomology}
	Let $A$ be a ring, $\underline{a}=a_1,\dots, a_r\in A$. For each $1\leq i\leq r$, $C^i(\underline{a})$ is the complex defined by following equations;
	\[C^i(\underline{a})=\sum_{\# I= i}A_{a_I}e_I,\]
	\[d^i\colon C^i(\underline{a})\to C^{i+1}(\underline{a});e_I\mapsto\sum_{j=1}^n e_I\wedge e_j.\]
	And for $i=0$, defined by $C^0(\underline{a})=A, d^0\colon a\mapsto\sum_{j=1}^r ae_j$ so $C^\bullet(\underline{a})$ form a complex. It called \textbf{\v{C}ech complex}. $\check{H}^i(\underline{a})$ denote the cohomology of this complex and it is called \textbf{\v{C}ech cohomology}.
\end{defi}

For $A$-module $M$, we define $C^\bullet(\underline{a},M)\coloneqq C^\bullet(\underline{a})\otimes M$. Here  $\check{H}^i(\underline{a},M)$ denote the $i$-th cohomology of $C^i(\underline{a},M)$. This form complex, denoted by $\check{H}^\bullet(\underline{a},M)$.

\begin{prop}\label{prop:cech is delta}
	Let $A$ be a ring, for each $a_1,\dots,a_r\in A$, $\check{H}^\bullet(\underline{a},-)$ is a $\delta$-functor.
\end{prop}

\begin{proof}
	Consider an exact sequence $\ses{M_1}{M_2}{M_3}$ of $A$-modules.
	Since $C^\bullet(\underline{a},M)=C^\bullet(\underline{a})\otimes M$ and each component of the \v{C}ech complex is a flat $A$-modules, the following sequence of complexes is exact;
	\[\ses{C^\bullet(\underline{a},M_1)}{C^\bullet(\underline{a},M_2)}{C^\bullet(\underline{a},M_3)}\]
	then there are connection morphisms. So $\check{H}^\bullet(\underline{a},-)$ is $\delta$-functor.
\end{proof}

For the result we want, we need to look at the relationship between \v{C}ech complex and Koszul complex.


For $\underbar{a}=a_1,\dots,a_r\in A$, let $\{e_i\}$ be the standard basis of a free $A$-module $A^r$. $f\colon A^r\to A;e_i\mapsto a_i$ induces a chain complex $K_\bullet(\underbar{a})$. In other words $K_\bullet(\underbar{a})$ is the complex defined by the following equation;
\[K_i(\underline{a})=\bigwedge^i A^r,\]
\[d_i\colon K_i(\underline{a})\to K_{i-1}(\underline{a});x_1\wedge\dots\wedge x_i\mapsto \sum_{j=1}^i(-1)^{j+1} f(x_j) x_1\wedge\dots\wedge\widehat{x}_j\wedge\dots\wedge x_i.\]
Note that $K_\bullet(\underline{a})$ does not depend on the order of $a_i$ .

We get a co-chain complex $K^\bullet(\underbar{a})$ from a contravariant functor $\hom(-,A)$;
\[K^\bullet(\underline{a})\colon \begin{tikzcd}
	0\nxcell A\nxcell\hom(K_1(\underline{a}),A)\nxcell\cdots
\end{tikzcd}.\]

$K^\bullet(\underbar{a})$ is called a Koszul complex. For each $A$-module $M$, $K^\bullet(\underline{a},M)=K^\bullet(\underbar{a})\otimes M$. Here $H^i(\underbar{a},M)$ denote the cohomology of Koszul complex.
\begin{lem}
	Let $A$ be a ring and $\underline{a}=a_1,\dots,a_r\in A$. For each $1\leq i\leq r$;
	\[\varphi^i\colon K^i(\underline{a})\to C^i(\underline{a});(e_I)^\ast\mapsto(1/a_I)e_I\]
	is a morphism of complexes.
\end{lem}

\begin{proof}
	Let $\delta^i$ be the derivative in the Koszul complex. Then;
	\[\delta^i(e_I^\ast)(e_J)=\begin{cases}
		a_j&(j\not\in I, J=I\cup\{j\})\\
		0&(\text{otherwise})
	\end{cases}\]
	So;
	\[\varphi^{i+1}\circ\delta^i(e_I^\ast)=\sum_{j\not\in I}\frac{a_j}{a_Ia_j} e_I\wedge e_j=\sum_{j\not\in I} \frac{1}{a_I} e_I\wedge e_j\]
	is equal to $d^i\circ\varphi^i(e_I^\ast)$.
\end{proof}

For any pair $n\leq m$;
\[\varphi_{mn}^\bullet\colon K^\bullet(\underline{a}^n)\to K^\bullet(\underline{a}^m)
;(e_I)^\ast\mapsto(a_I)^{m-n}(e_I)^\ast\]
then $\{K^\bullet(\underline{a}^n)\}_{n\in\N}$ is an inductive system.

\begin{prop}
	Let $A$ be a ring, $a_1,\dots,a_r\in A$. Then;
	\[\ilim K^\bullet(\underline{a}^n)\cong C^\bullet(\underline{a}).\]
\end{prop}

\begin{proof}
	We define $\varphi_n^\bullet\colon K^\bullet(\underline{a}^n)\to C^\bullet(\underline{a}^n)=C^\bullet(\underline{a})$ in the same way as above lemma. Then $\varphi_m^\bullet\circ\varphi_{nm}^\bullet=\varphi_n^\bullet$ where $n\leq m$. So we have $\varphi\colon \ilim K^\bullet(\underline{a}^n)\to C^\bullet(\underline{a})$. Each element of $C^i(\underline{a})$ is represented finite sum of $(b_I/a_I^{n_I})e_I$, so it can be displayed as $\sum (1/a_I^n)b_Ie_I$ by taking the maximum of $n$ and replacing $b_I$. Then it is image of $\sum(b_Ie_I)\in K^i(\underline{a}^n)$ , so $\varphi$ is surjective.
	
	Secondly, we show $\varphi$ is injective. Assume $\varphi^i_n(x)=0$ for $x\in K^i(\underline{a}^n)$. If $x=\sum b_Ie_I^\ast$ then $\varphi^i_n(x)=\sum (b_I/a_I^n) e_I^\ast =0$, so $b_I/a_I^n=0$ in $A_{a_I^n}$. Therefore if we take a sufficiently large $l$, $a_I^lb_I=0$. So  $\varphi_{nm}^l(x)=0$ by increasing $l$ if necessary, then $\varphi$ is injective.
\end{proof}

Since the functor of taking the inductive limit is exact, the following Corollary follows.
\begin{cor}
	Let $A$ be a ring, $a_1,\dots,a_r\in A$. For each $A$-module $M$;
	\[\check{H}^i(\underline{a},M)\cong\ilim H^i(\underline{a}^n,M).\]
\end{cor}
\section{Weakly proregular sequences}

\begin{defi}
	Let $\mathscr{A}$ be an Abelian category, $(X_n,\varphi_{mn})$ a projective system in $\mathscr{A}$. $(X_n)$ is said to be \textbf{essentially zero} or \textbf{pro-zero} if for each $n$, there is $m\geq n$ such that $\varphi_{mn}\colon X_m\to X_n$ is zero map.
\end{defi} 

Obviously, if $(X_n)$ is essentially zero then $\plim X_n=0$.

\begin{prop}
	Let $\mathscr{A}$ be an Abelian category. We consider a exact sequence of projective system in $\mathscr{A}$;
	\[\ses[(f_n)][(g_n)]{(X_n)}{(Y_n)}{(Z_n)}.\]
	Then $(Y_n)$ is essentially zero if and only if the other two are essentially zero.
\end{prop}

\begin{proof}
	$(\Longrightarrow)$ Trivial. $(\Longleftarrow)$ For each $n$, there is $m\geq n$ such that $X_m\to X_n$ since $(X_n)$ is essentially zero. Similarly there is $l\geq m$ such that $Z_l\to Z_m$ is zero map, then we have the following commutative diagram with the rows are exact;
	\[\begin{tikzcd}
		0\nxcell X_{l}\darrow[]\nxcell[f_{l}] Y_{l}\darrow[\varphi_{lm}]\nxcell[g_{l}]Z_{l}\darrow[0]\nxcell 0\\
		0\nxcell X_m\darrow[0]\nxcell[f_m]Y_m\darrow[\varphi_{mn}]\nxcell[g_m] Z_m\darrow[]\nxcell0\\
		0\nxcell X_n\nxcell[f_n]Y_n\nxcell[g_n]Z_n\nxcell0
	\end{tikzcd}\]
	So we get $\varphi_{ln}=\varphi_{mn}\circ\varphi_{lm}=0$ by easy diagram chasing.
\end{proof}

We use same symbols as before for the Koszul and \v{C}ech complexes. Note that $(K_i(\underline{a}^n))_{n\in\N}$ is a projective system defined by $K_i(\underline{a}^m)\to K_i(\underline{a}^n);e_I\mapsto a_I^{m-n}e_I~(m\geq n)$.

\begin{defi}
	Let $A$ be a ring. $a_1,\dots,a_r\in A$ is called a \textbf{weakly proregular sequence} if for each $1\leq i\leq r$, the projective system $\{H_i(\underline{a}^n)\}$ is essentially zero.
\end{defi} 

The property of being weakly proregular does not depend on the order by the definition.

\begin{prop}\label{prop: w.p.s iff cech is effaceable}
	Let $A$ be a ring, $\underline{a}=a_1,\dots,a_r\in A$. $\underline{a}$ is a weakly proregular sequence if and only if $\check{H}^\bullet(\underline{a},-)$ is an effaceable $\delta$-functor. 
\end{prop}

\begin{proof}
	Assume that $\underline{a}$ is a weakly proregular sequence. Let $I$ be an injective module. Now there is a isomorphism $H^i(\underline{a}^n,I)\cong \hom(H_i(\underline{a}^n),I)$ since $K^\bullet(\underline{a}^n,I)=\hom(K_\bullet(\underline{a}^n),I)$ and $\hom(-,I)$ is an exact functor. For each $n\geq0$, there is $m\geq n$ such that $H_i(\underline{a}^m)\to H_i(\underline{a}^n)$ is zero map since $H_i(\underline{a}^n)$ is essentially zero. So $\check{H}^i(\underline{a},I)=\ilim H^i(\underline{a}^n,I)=0$.
	
	Secondly, assume that $\check{H}^\bullet(\underline{a},-)$ is effaceable $\delta$-functor. For each $n\geq 0$, we have injective module $I$ and injection $\varepsilon\colon H_i(\underline{a}^n)\to I$. Then there is $m\geq n$ such that;
	\[\begin{tikzcd}
		H_i(\underline{a}^m)\nxcell H_i(\underline{a}^n)\nxcell[\varepsilon]I
	\end{tikzcd}\]
	is zero map by $\varepsilon\in H^i(\underline{a}^n,I)$ and $\ilim H^i(\underline{a}^n,I)=0$.
\end{proof}

Then \v{C}ech cohomology is the derived functor of $\check{H}^0(\underline{a},-)$ if $\underline{a}$ is a weakly proregular sequence. So the next question of interest is when is a sequence weakly proregular?

\begin{defi}
	Let $A$ be a ring, $\underline{a}=a_1,\dots,a_r\in A$. $\underline{a}$ is called \textbf{proregular sequence} if for each $1\leq i\leq r$ and $n>0$, there is $m\geq n$ such that $((a_1^m,\dots,a_{i-1}^m):a_i^mA)\subset ((a_1^n,\dots,a_{i-1}^n):a_i^{m-n}A)$.
\end{defi}

Note that a regular sequence is proregular and this phenomenon is characteristic of non-Noetherian rings.

\begin{prop}
	Let $A$ be a Noetherian ring. For each $a_1,\dots,a_r\in A$ is a proregular sequence.
\end{prop}

\begin{proof}
	Let $J_m=((a_1^m,\dots,a_{i-1}^m):a_i^mA), I_{n,m}=((a_1^n,\dots,a_{i-1}^n):a_i^{m-n}A)$. Now $\{I_{n,m}\}_{m\geq n}$ is an  ascending chain of ideals, then there is $m_0\geq n$ such that for each $m\geq m_0,  I_{n,m_0}=I_{n,m}$. Let $m=m_0+n$, then for each $a\in J_{m_0}$, $aa_i^{m-n}=aa_i^{m_0}\in(a_1^{m_0},\dots,a_{i-1}^{m_0})\subset(a_1^n,\dots,a_{i-1}^n)$. So $a\in I_{n,m}=I_{n,m_0}$.
\end{proof}

\begin{prop}
	Let $A$ be a ring. A proregular sequence is weakly proregular.
\end{prop}

\begin{proof}
	We use induction on $r$. When $r=1$, let $a\in A$ be proregular. Then for each $n>0$, there is $m\geq n$ such that $\ann a^m\subset\ann a^{m-n}$. So $(H_1(a^n))$ is essentially zero since $H_1(a^n)=\ann a^n$. Now we assume that claim up to $r-1$. The exact sequence of complexes;
	\[\begin{tikzcd}
		0\nxcell{K_\bullet(a_1^n,\dots,a_{r-1}^n)}\nxcell{K_\bullet(a_1^n,\dots,a_r^n)}&\phantom{a}\\
		&\phantom{\mbox{$K_\bullet(a_1^n,\dots,a_r^n)$}}\nxcell K_\bullet(a_1^n,\dots,a_{r-1}^n)(-1)\nxcell0
	\end{tikzcd}\]
	induces the exact sequence of homology;
	\settowidth{\masyulengtha}{$H_1(a_1^n,\dots,a_r^n)$}
	\[\begin{tikzcd}
		&\makebox[\masyulengtha]{$\cdots$}\nxcell H_i(a_1^n,\dots,a_{r-1}^n)\nxcell[(-1)^ia_r^n]\\
		H_i(a_1^n,\dots,a_{r-1}^n)\nxcell H_i(a_1^n,\dots,a_r^n)\nxcell H_{i-1}(a_1^n,\dots,a_{r-1}^n)\nxcell[(-1)^{i-1}a_r^n]\\
		H_{i-1}(a_1^n,\dots,a_{r-1}^n)\nxcell\makebox[\masyulengtha]{$\cdots$}
	\end{tikzcd}\]
	Then we have the following exact sequence;
	\[\begin{tikzcd}
		0\nxcell{H_0(a_r^n,H_i(a_1^n,\dots,a_{r-1}^n))}\nxcell{H_i(a_1^n,\dots,a_r^n)}&\phantom{a}\\
		\phantom{0}\nxcell{H_1(a_r^n,H_{i-1}(a_1^n,\dots,a_{r-1}^n))}\nxcell0
	\end{tikzcd}\]
	and this induces the exact sequence of projective systems. The first projective system is essentially zero by the assumption of induction. Also for each $i>1$, the third system is essentially zero since $H_1(a_r^n,H_{i-1}(a_1^n,\dots,a_{r-1}^n))=\mkset{x\in H_{i-1}(a_1^n,\dots,a_{r-1}^n)}{a_r^nx=0}$. If $i=1$, the system with ;
	\[H_1(a_r,H_0(a_1^n\dots,a_{r-1}^n))=\mkset{x\in H_{0}(a_1^n,\dots,a_{r-1}^n)}{a_r^nx=0}\]
	is essentially zero since $\underline{a}$ is proregular. So this completes the proof by induction.
\end{proof}

\begin{cor}
	Let $A$ be a Noetherian ring. For each $\underline{a}=a_1,\dots,a_r\in A$, $\underline{a}$ is weakly proregular.
\end{cor}
\section{Local cohomology}

Let $A$ be a ring and $I$ an ideal of $A$. The functor $\Gamma_I$ is defined by;
\[\Gamma_I(M)=\mkset{x\in M}{\text{$I^nx=0$ for some $n \geq 0$}}\]
for an $A$-module $M$. Note that $\Gamma_I(M)=\ilim\hom_A(A/I^n,M)$ and this isomorphism is functorial in $M$. By the definition, $\Gamma_I$ is a left exact functor.

\begin{defi}
	Let $A$ be a ring and $I$ an ideal of $A$. $H^i_I(-)$ denote the derived functor of $\Gamma_I(-)$ and it is called \textbf{local cohomology}.
\end{defi}

Note that $H^i(M)\cong\ilim\Ext^i(A/I^n,M)$. We summarize the relationship between local cohomology and \v{C}ech cohomology. First, we note that the 0-th parts of each cohomology are naturally isomorphic.

\begin{lem}
	Let $A$ be a ring, $\underline{a}=a_1,\dots,a_r\in A$, and $I=(a_1,\dots,a_r)$. For each $A$-module $M$;
	\[\Gamma_I(M)\cong\check{H}^0(\underline{a},M).\]
\end{lem}

\begin{proof}
	Here $\check{H}^0(\underline{a},M)$ is the kernel of
	\[M\to\bigoplus_{i=1}^r M_{a_i}e_i;x\mapsto (x/1)e_i .\] 
	Then for each $x\in\check{H}^0(\underline{a},M)$ and $1\leq i\leq r$, there is $n_i\geq 0$ such that $a_i^{n_i}x=0$. So we have $x\in\Gamma_I(M)$. Similarly the converse is true, so they are equal as submodules of $M$.
\end{proof}

With the preparations we have made above, we can prove the results we have been aiming for.

\begin{thm}\label{thm:goal}
	Let $A$ be a ring, $\underline{a}=a_1,\dots,a_r\in A$ and $I=(a_1,\dots,a_r)$. $\underline{a}$ is a weakly proregular sequence if and only if for any $i$ and $A$-module $M$, $H^i_I(M)\cong\check{H}^i(\underline{a},M)$.
\end{thm}

\begin{proof}
	Assume that $\underline{a}$ is weakly proregular. $\check{H}^\bullet(\underline{a},-)$ is a $\delta$-functor by \ref{prop:cech is delta}. Moreover  $\check{H}^\bullet(\underline{a},-)$ is universal by \ref{prop: w.p.s iff cech is effaceable} and \ref{prop:effaceable->universal}. So $H^i_I(M)\cong\check{H}^i(\underline{a},M)$ by above lemma. The converse is true by \ref{prop: w.p.s iff cech is effaceable}.
\end{proof}
\subsection*{Acknowledgments}
I thank Hisanori Ohashi, and Yuya Matsumoto 
for helpful comments and discussions.


\begin{bibdiv}
	\begin{biblist}
		\bib{BrunsHerzog}{book}{
			author 		={Bruns, W.},
			author		={Herzog, J.},
			title		={Cohen--Macaulay Rings (Revised ed.)},
			publisher 	={Cambridge Univ. Press},
			year		={1997}
		} 
		\bib{GreenleesMay}{article}{
			title 		={Derived functors of I-adic completion and local homology},
			journal 	={Journal of Algebra},
			volume 		={149},
			number 		={2},
			pages 		={438--453},
			year 		={1992},
			author 		={Greenlees, J. P. C.},
			author		={May, J. P.}
		}
		\bib{Tohoku}{article}{
			author 		={Grothendieck, A.},
			title 		={Sur quelques points d'algèbre 	homologique, I},
			journal 	={Tohoku Math. J.},
			number		={2},
			pages 		={119--221},
			volume 		={9},
			year 		={1957}
		}
		\bib{Grothendieck}{book}{
			title		={Local cohomology, notes by R. Hartshorne},
			author		={Grothendieck, A.},
			year		={1966},
			publisher	={Springer}
		}
		\bib{Hartshorne}{book}{
			author 		={Hartshorne, R.},
			title 		={Algebraic Geometry},
			publisher 	={Springer},
			year 		={1977},
		}
		\bib{Schenzel}{article}{
			author={Schenzel, P.},
			title={Proregular sequences, local cohomology, and completion},
			journal={Math. Scand.},
			volume={92},
			year={2003},
			pages={161--180},
		}
	\end{biblist}
\end{bibdiv}
\end{document}